\documentclass[a4paper,11pt,oneside,openany,reqno]{amsart}

\usepackage{latexsym,amsmath,amssymb,amsfonts}
\usepackage{bm}
\usepackage[dvipdfmx]{graphicx}
\usepackage{subfigure}
\usepackage{verbatim}
\usepackage{wrapfig}
\usepackage{ascmac}
\usepackage{cases}
\usepackage{amsthm}
\usepackage{color}

\usepackage{comment}

  \makeatletter
    
    \@addtoreset{equation}{section}
  \makeatother
\newtheorem{theo}{Theorem}[section]

\newtheorem{lemm}[theo]{Lemma}
\newtheorem{cor}[theo]{Corollary}
\newtheorem{rem}[theo]{Remark}

\newcounter{c}
\newcounter{d}
\newcounter{b}
\newcommand{\cc}[1][]{\refstepcounter{c}#1\arabic{c}}

\newcommand{\res}{\mathop{\hbox{\vrule height 7pt width 
0.5pt depth 0pt \vrule height 0.5pt width 6pt depth 0pt}}\nolimits}

\newcommand{\Na}{\mathbb{N}} 
\newcommand{\R}{\mathbb{R}} 

\newcommand{\e}{\varepsilon} 
\newcommand{\spt}{\mathrm{spt}} 
\newcommand{\di}{\mathrm{div}} 
\newcommand{\Ha}{\mathcal{H}} 
\renewcommand{\L}{\mathcal{L}} 

\title[Prescribed mean curvature equation]{The Dirichlet problem for a prescribed mean curvature equation}
\author{Yuki Tsukamoto}
\address{Department of Mathematics, Tokyo Institute of Technology,
	152-8551, Tokyo, Japan}
\email{tsukamoto.y.ag@m.titech.ac.jp}
\date{}

\begin{document}
\maketitle
\begin{abstract}
	We study a prescribed mean curvature problem where 
	we seek a surface whose mean curvature vector coincides with the normal 
	component of a given vector field. 
	We prove that the problem has a solution near a graphical minimal surface if 
	the prescribed vector field is sufficiently small in a dimensionally sharp 
	Sobolev norm.
\end{abstract}

	\section{Introduction}
	In this paper, we consider the following prescribed mean 
	curvature problem with the Dirichlet condition,
 \begin{align}
	\begin{cases}
		\di\left( \frac{\nabla u}{\sqrt{1+|\nabla u|^2}} \right) =
		H(x, u(x), \nabla u(x))  \quad
		\mathrm{in} \  \Omega, \\
		u= \phi \quad \mathrm{on} \  \partial \Omega, \label{it2}
	\end{cases}
\end{align}
	where $\Omega $ is a bounded domain in $\R^n$.
  The function $H(x,t,z): \R^{n} \times \R \times  \R^{n} \to \R$ is given 
  and we seek a solution $u$ satisfying \eqref{it2}. Since the left-hand
  side is the mean curvature of the graph of $u$, \eqref{it2} is a prescribed
  mean curvature equation whose prescription depends on the location of the graph
  as well as the slope of the tangent space.

	Prescribed mean curvature problems in a wide variety of formulation have been 
	studied by numerous researchers. 
	In the most classical case of $H=H(x)$, (\ref{it2}) has a solution
	if $H$ and $\phi$ have a suitable regularity and
	the mean curvature of $\partial \Omega$ satisfies a certain geometric condition 
	(see \cite{GM74, G83,H00, J68, S69}, for example).
	Giusti \cite{G78} determined a necessary and sufficient condition that 
	a prescribed mean curvature problem without boundary conditions has solutions.
	In the case of $H=H(x,t)$, Gethardt \cite{GC74} constructed $H^{1,1}$ solutions, and
	Miranda \cite{M74} constructed BV solutions.
	In those papers, assumptions of the boundedness $|H|< \infty$
	and the monotonicity $\frac{\partial H}{\partial t} \geq 0$ play an important role.
	If $|H|< \Gamma$ where $\Gamma$ is determined by $\Omega$, there exist solutions, and
	the uniqueness of solutions is guaranteed by the monotonicity, that is, $\frac{\partial H}{\partial t} \geq 0$. 
	Under the assumptions of boundedness, monotonicity and the convexity of $\Omega$, Bergner \cite{B08} solved the Dirichlet problem
	 in the case of $H=H(x, u, \nu(\nabla u) )$ using the Leray-Schauder  fixed point theorem.
	 Here, $\nu$ is the unit normal vector of $u$, that is,
	 $\nu(z)=\frac{1}{\sqrt{1+|z|^2}} (z, -1)$.
	 For the same problem as \cite{B08}, 
	 Marquardt \cite{M10} gave a condition on $\partial\Omega$ depending on $H$
	 which guarantees the existence of solution even for non-convex domain $\Omega$. 
	 
The motivation of the present paper comes from a singular perturbation problem
studied in \cite{T19}, where one considers the following problem on a domain $\tilde \Omega\subset
\mathbb R^{n+1}$,
\begin{align}
-\varepsilon\Delta \phi_{\varepsilon}+\frac{W'(\phi_{\varepsilon})}{\varepsilon}=\varepsilon\nabla\phi_{\varepsilon}\cdot f_{\varepsilon}.
\end{align}
Here, $W$ is a double-well potential, for example $W(\phi)=(1-\phi^2)^2$ and 
$\{f_{\varepsilon}\}_{\varepsilon>0}$ are given vector fields uniformly bounded in the Sobolev norm of 
$W^{1,p}(\tilde\Omega)$, $p>\frac{n+1}{2}$. In \cite{T19}, we proved under a natural assumption 
\begin{align}
\int_{\tilde \Omega}\big(\frac{\varepsilon|\nabla\phi_{\varepsilon}|^2}{2}
+\frac{W(\phi_{\varepsilon})}{\varepsilon}\big)\,dx+
\|f_{\varepsilon}\|_{W^{1,p}(\tilde\Omega)}\leq C
\end{align}
that the interface
$\{\phi_{\varepsilon}=0\}$ converges locally in the Hausdorff distance to a suface whose mean curvature $H$
is given by $f\cdot \nu$ as $\varepsilon\rightarrow 0$. Here, $f$ is the 
weak $W^{1,p}$ limit of $f_{\varepsilon}$. If the surface is represented locally as a graph of a function $u$
over a domain $\Omega\subset\mathbb R^{n}$, the corresponding relation between
the mean curvature and the vector field is expressed as
\begin{align}
\di\left( \frac{\nabla u}{\sqrt{1+|\nabla u|^2}} \right) =\nu(\nabla u(x)) \cdot f(x,u(x))  \quad
\mathrm{in} \  \Omega, \label{qe1}
\end{align}
  where $f \in W^{1,p}(\Omega \times \R ; \R^{n+1})$ with $p>\frac{n+1}{2}$.
  Note that $f$ is not bounded in $L^{\infty}$ in general, unlike the 
  cases studied in \cite{B08,M10}. In this paper, we establish the well-posedness
  of the perturbative problem including \eqref{qe1} which has a $W^{1,p}$ norm control on the right-hand
  side of the equation. 
  The following theorem is the main result of this paper.
 
  \begin{theo} \label{mt2}
  	Let $\Omega$ be a $C^{1,1}$ bounded domain in $\R^n$ and fix constants $\e>0$, 
	$\frac{n+1}{2}<p< n+1$ and $q=\frac{np}{n+1-p}$. Suppose that $h \in W^{2,\infty}(\Omega)$ satisfies
  	the minimal surface equation, that is,
  	\begin{align}
  		\di\left( \frac{\nabla h}{\sqrt{1+|\nabla h|^2}} \right) =0. \label{t6e1a}
  	\end{align}
  	Then there exists a constant $\delta_1>0$ which depends only on $n$, $p$, $\Omega$,
  	$\|h\|_{W^{2,\infty}(\Omega)}$ and $\e$ with the following property.
Suppose that 
  	$G \in W^{1,p}(\Omega \times \R )$ and 
  	$\phi \in W^{2,q}(\Omega)$
	satisfy
  	\begin{align}
  		\|G\|_{W^{1,p}(\Omega \times \R )} +\|\phi \|_{W^{2,q}(\Omega)} \leq \delta_1, \label{qe2}
  	\end{align}
   and a measurable function $H(x,t,z):\R^n \times \R \times \R^{n}
  	\to \R$ is such that 
  	$H(x,\cdot,\cdot)$ is a continuous function for a.e.~$x \in \Omega$, and
	for all $ (t,z) \in \R \times \R^{n}$,
  	\begin{align}
  		|H(x,t,z)| \leq |G(x,t)| \quad \mathrm{for} \   a.e. \ x \in \Omega.
  	\end{align} 
  	Then, there exists a function $u \in W^{2,q}(\Omega)$ 
  	such that $ u-h-\phi \in W^{1,q}_0(\Omega)$ and
  	\begin{align}
  		\di\left( \frac{\nabla u}{\sqrt{1+|\nabla u|^2}} \right)& =
  		H(x, u(x), \nabla u(x))  \quad
  		\mathrm{in} \  \Omega, \label{t6e1d} \\
  		\|u-h\|_{W^{2,q}(\Omega)}&<\e. \label{t6e1e}
  	\end{align}
  \end{theo}
 The claim proves that there exists a solution of (\ref{it2}) in a neighbourhood 
 of any minimal surfaces 
  if $H$ and $\phi$ are sufficiently small in these norms. In particular, if we
  take $H(x,t,z)=\nu(z)\cdot f(x,t)$ and $G(x,t)=|f(x,t)|$, where $\|f\|_{W^{1,p}
  (\Omega\times\R)}$ is sufficiently small, above conditions on $G$ and $H$ are satisfied and
  we can guarantee the existence of a solution for \eqref{it2} nearby the given
  minimal surface (see Corollary \ref{maincor}).  
The method of proof is as follows.
We prove that linear elliptic equations have a unique solution in $W^{2,q}(\Omega)$ and
the norm of this solution is controlled by $G$ and $\phi$. When (\ref{qe2}) is
satisfied, there exist a suitable function space $\mathcal{A}$ and a mapping 
$T:\mathcal{A}\to \mathcal{A}$, and a fixed point of $T$ is a solution 
of Theorem \ref{mt2}. We show that $T$ satisfies assumptions of the Leray-Schauder fixed point theorem, and Theorem \ref{mt2} follows.
  	
	\section{Proof of Theorem \ref{mt2}}
	Throughout the paper, $\Omega$ is a bounded domain in $\R^n$ with $C^{1,1}$ boundary
	$\partial \Omega$.
We define functions $A_{ij}: \R^n \to \R$ ($i,j=1,\ldots,n$)
\[
A_{ij}(z) := \frac{1}{\sqrt{1+|z|^2}}\left( \delta_{ij}-\frac{z_i  z_j}{1+|z|^2}\right)
\]
and the operator
\[
L[z](u) :=A_{ij}(z) u_{x_i x_j}(x) \quad \mathrm{for} 
\ \mathrm{any} \ u\in W^{2,1}(\Omega),
\]
where we omit the summation over $i,j=1,\ldots,n$. 
By the Cauchy--Schwarz inequality, for any $\xi \in \R^n$,
\begin{align}
	A_{ij}(z)\xi_i \xi_j  &=\frac{1}{\sqrt{1+|z|^2}}\left( \delta_{ij}-\frac{z_i z_j}{1+|z| ^2}\right) \xi_i \xi_j \nonumber \\
	&=  \frac{1}{\sqrt{1+|z|^2}}\left[ \xi_i^2  -\left( \frac{z_i}{\sqrt{1+|z|^2}} \xi_i \right)^2 \right] \nonumber\\
	&\geq \frac{1}{\sqrt{1+|z|^2}}\left[|\xi|^2 -\left( \frac{|z|^2}{1+|z|^2}\right)|\xi|^2
\right] \nonumber \\
	&=\frac{1}{(1+|z|^2)^\frac{3}{2}} |\xi|^2. \label{t0e1}
\end{align}
Hence, as is well-known, the operator $L[z]$ is elliptic.

\begin{theo} 	\label{th1}
	Suppose that we are given $v \in C^{1,\alpha}(\bar{\Omega})$ where $0<\alpha<1$, $f \in L^q(\Omega)$ and
	$\phi \in W^{2,q}(\Omega)$ where $q>n$.
	Then there exists a unique function $u \in W^{2,q}(\Omega)$ such that
	\begin{align}
	\begin{cases}
	L[\nabla v](u) =f(x) \quad
	\mathrm{in} \  \Omega, \\
	u-\phi \in W^{1,q}_0(\Omega). \label{t1e1}
	\end{cases}
	\end{align}
	Moreover, there exists a constant $c_{\cc \label{t1c1}}$ which depends only on $n$, $q$, $\Omega$ and $\| v \|_{C^{1,\alpha}(\bar{\Omega})}$ such that
	\begin{align}
	\| u\|_{W^{2,q}(\Omega)} \leq c_{\ref{t1c1}}(\|f\|_{L^q(\Omega)}
	+\|\phi\|_{W^{2,q}(\Omega)}) . \label{t1e2}
	\end{align}
\end{theo}
\begin{proof}
	By (\ref{t0e1}), for any $\xi \in \R^n$,
	\begin{align}
	A_{ij}(\nabla v) \xi_i \xi_j \geq  
	\frac{1}{(1+\| v \|_{C^{1,\alpha}(\bar{\Omega})}^2)^\frac{3}{2}} |\xi|^2
	=: \lambda |\xi|^2. \label{t1e2a}
	\end{align}
where the constant $\lambda$ depends only on $\| v \|_{C^{1,\alpha}(\bar{\Omega})}$.
    Since each $A_{ij}$ is a smooth function of $\nabla v$, there exists a constant $\Lambda$ which depends 
    only on $\| v \|_{C^{1,\alpha}(\bar{\Omega})}$ such that
    \begin{align}
    \| A_{ij}(v) \|_{C^{0,\alpha}(\bar{\Omega})} \leq  \Lambda
    \quad \mathrm{for \  all \ } i,j \in \{ 1, \cdots ,n \}. \label{t1e2b}
    \end{align}
By (\ref{t1e2a}) and (\ref{t1e2b}),
there exists a unique solution $u \in W^{2,q}(\Omega) $ satisfying (\ref{t1e1}) using \cite[Theorem 9.15]{G83}.
	Using \cite[Theorem 9.13]{G83}, there exists a constant  $c_{\cc 
		\label{t1c2}}$ which depends only on $n$, $q$, $\Omega$, $\lambda$ and $\Lambda$ such that
	\begin{align}
	\| u\|_{W^{2,q}(\Omega)} \leq c_{\ref{t1c2}}(\|u\|_{L^q(\Omega)} + \|f\|_{L^q(\Omega)} +\|\phi\|_{W^{2,q}(\Omega)}). \label{t1e3}
	\end{align}
		Using the Aleksandrov maximum principle \cite[Theorem 9.1]{G83}, there exists a constant $c_{\cc \label{t1c3}}$ which depends only on $n$, $\Omega$ and $\lambda$ such that
	\begin{align}
	\|u\|_{L^\infty (\Omega)} &\leq \sup_{x \in {\partial \Omega}} |u| +c_{\ref{t1c3}}  \| f\|_{L^n(\Omega)} \nonumber \\
	&=\sup_{x \in {\partial \Omega}} |\phi| +c_{\ref{t1c3}}  \| f\|_{L^n(\Omega)}.
	 \label{t1e4}
	\end{align}
	By the H\"{o}lder and Sobolev inequalities,
	\begin{align}
	\|u\|_{L^q(\Omega)} &\leq c\|u\|_{L^\infty (\Omega)} \nonumber \\
	&\leq c( \sup_{x \in {\partial \Omega}} |\phi| +  \| f\|_{L^n(\Omega)})  
	\nonumber \\
	&\leq c( \|\phi \|_{L^\infty(\Omega)} +  \| f\|_{L^n(\Omega)})  
	\nonumber \\
	& \leq  c_{\ref{t1c3a}}(\|f\|_{L^q(\Omega)}+\|\phi\|_{W^{2,q}(\Omega)})  ,	\label{t1e5}
	\end{align}
	where $c_{\cc \label{t1c3a}}$ depends only on $n$, $q$ and $\Omega$.
	By (\ref{t1e3}) and (\ref{t1e5}), there exists a constant $c_{\ref{t1c1}}$ which depends only on $n$, $q$,
	 $\Omega$, $\lambda$ and $\Lambda$ such that
	\begin{align}
	\| u\|_{W^{2,q}(\Omega)} \leq c_{\ref{t1c1}} (\|f\|_{L^q(\Omega)}
	+\|\phi\|_{W^{2,q}(\Omega)}).
	\end{align}
	Thus this theorem follows.
\end{proof}

To proceed, we need the following theorem  (see \cite[Theorem 5.12.4]{WP72}).
\begin{theo} \label{th2}
	Let $\mu$ be a positive Radon measure on $\mathbb{R}^{n+1}$ satisfying
	\[
	K(\mu):=\sup_{B_r(x)\subset\R^{n+1}} \frac{1}{r^{n}} \mu(B_r(x))< \infty.
	\]
	Then there exists a constant $c(n)$ such that
	\[
	\left|\int_{\mathbb{R}^{n+1}} \phi  \,d\mu \right| \leq c(n) K(\mu) \int_{\mathbb{R}^{n+1}} |\nabla \phi| \,d\mathcal{L}^{n+1}
	\]
	for all $\phi \in C^1_c(\mathbb{R}^{n+1})$. \label{t2e1}
\end{theo} 

\begin{lemm} \label{lm3}
Suppose that $v \in W^{1,\infty}(\Omega)$ with $\|v \|_{W^{1,\infty}(\Omega)} \leq V$ 
and $G \in W^{1, p}(\Omega \times \R)$ 
	where $\frac{n+1}{2} <p < n+1$. Suppose that $q=\frac{np}{n+1-p}(>n)$. Then there exists a constant $c_{\cc \label{t3c1}}$ which depends only on
	$n$, $p$, $\Omega$ and $V$ such that
	\begin{align}
		 \|G(\cdot,v(\cdot))\|_{L^{q}(\Omega)} \leq c_{\ref{t3c1}} 
			\|G\|_{W^{1,p}(\Omega \times  \R)}.
	\end{align}
\end{lemm}
\begin{proof}
	Define
	\[
	\Gamma := \{(x,v(x)) \in \Omega \times \R \}.
	\]
	A set $B^{n}_r(x)$ is the open ball with center $x$ and radius $r$ in $\R^n$.
	In the following, $\Ha^n$ denotes the $n$-dimensional Hausdorff measure in 
	$\R^{n+1}$ and 
	$\Ha^{n} \res_{\Gamma}$ is a Radon measure defined by
	\[
	\Ha^{n} \res_{\Gamma}(A):= \Ha^{n} (A\cap \Gamma) \quad 
	\mathrm{for \ all \ } A\subset \R^{n+1}.
	\]
	Then the support satisfies in particular 
	$\spt \Ha^{n} \res_{\Gamma} \subset \Omega \times (-2V,2V)$.
	For any $B^{n+1}_r((x_0,x'_0))  \subset \R^{n+1}$ where $(x_0,x'_0) \in \R^n \times \R$,
	\begin{align}
		\frac{1}{r^{n}} \Ha^{n} \res_{\Gamma}(B^{n+1}_r((x_0,x'_0))) \leq 
		\frac{1}{r^{n}} \int_{B^{n}_r(x_0) \cap \Omega} \sqrt{1+ |\nabla v|^2} \ d \L^n \leq 
		(1+V)\omega_n .\label{t3e1}
	\end{align}
	Using the standard Extension Theorem, there exists a function
	$\tilde{G} \in W^{1, p}_0(\R^{n+1} )$ such that $\tilde{G}=G$ in $\Omega\times  (-2V,2V)$ and
	\begin{align}
		\|\tilde{G}\|_{W^{1,p}(\R^{n+1})} \leq c_{ \ref{t3c2}}
		\|G\|_{W^{1,p}(\Omega \times  (-2V,2V))},
	\end{align}
	where $c_{\cc \label{t3c2}}$ depends only on $n$, $p$, $\Omega$ and $V$.
	By Theorem \ref{th2} and smoothly approximating $\tilde{G}$,
	\begin{align}
		\int_{\Omega} |G(x,v(x))|^q & \leq \int_{\Omega} |\tilde{G}(x,v(x))|^q \sqrt{1+|\nabla v|^2} \nonumber \\ 
		&=   \int_{\Gamma}  |\tilde{G}(x,x_{n+1})|^q \ d\Ha^n \nonumber \\
		&\leq c(n,V) \int_{\R^{n+1}} |\nabla\tilde{G}| |\tilde{G}|^{q-1} \ d \L^{n+1}  \nonumber \\
		& \leq c(n,p,V) \|\nabla \tilde{G}\|_{L^p(\R^{n+1} )}
		\|\tilde{G}\|_{W^{1,p}(\R^{n+1} )}^{q-1} \nonumber \\
		&  \leq c(n,p,V) c_{\ref{t3c2}} \|G\|^q_{W^{1,p}(\Omega \times (-2V,2V))} \nonumber \\
		&\leq c(n,p,V) c_{\ref{t3c2}} \|G\|^q_{W^{1,p}(\Omega \times \R)}. \label{t3e2}
	\end{align}
	This lemma follows.
\end{proof}

We write the Leray-Schauder fixed point theorem needed later ( \cite[Theorem 11.3]{G83}).
\begin{theo} \label{th4} 
Let $T$ be a compact and continuous mapping of a Banach space $\mathcal{B}$ into itself,
and suppose that there exists a constant $M$ such that
\[
\|u \|_\mathcal{B}<M
\]
for all $u \in \mathcal{B}$. Then $T$ has a fixed point.
\end{theo}

We first prove Theorem \ref{mt2} in the case that $h=0$. 
 \begin{theo} \label{mt1}
	Assume that 
	$G \in W^{1,p}(\Omega \times \R )$ with $\frac{n+1}{2}<p< n+1$ and
	$\phi \in W^{2,q}(\Omega)$ with $q=\frac{np}{n+1-p}$.
	Then there exists a constant $\delta_2>0$ which depends only on $n$, $p$ and $\Omega$
	such that, if
	\begin{align}
	\|G\|_{W^{1,p}(\Omega \times \R )} +\|\phi \|_{W^{2,q}(\Omega)} \leq \delta_2,
	\end{align}
	then, for any measurable function $H(x,t,z):\R^n \times \R \times \R^{n}
	\to \R$ such that 
	$H(x,\cdot,\cdot)$ is a continuous function for a.e.~$x \in \Omega$ and
	\begin{align}
	|H(x,t,z)| \leq |G(x,t)| \quad \mathrm{for} \   a.e. \ x \in \Omega, \ 
	\mathrm{any} \  (t,z) \in \R \times \R^{n}, \label{t00e1}
	\end{align} 
	there exists a function $u \in W^{2,q}(\Omega)$ 
	such that $ u-\phi \in W^{1,q}_0(\Omega)$ and
	\begin{align}
	\di\left( \frac{\nabla u}{\sqrt{1+|\nabla u|^2}} \right) =
	H(x, u(x), \nabla u(x))  \quad
	\mathrm{in} \  \Omega. \label{t5e1}
	\end{align}
\end{theo}
\begin{proof}
	Define 
	\begin{align}
	\mathcal{A}:=\{v \in C^{1,\frac{1}{2}-\frac{n}{2q}}(\bar{\Omega})
	 ; \|v\|_{C^{1,\frac{1}{2}-
			\frac{n}{2q}(\bar{\Omega})}} \leq 1  \}.
	\end{align}
	By (\ref{t00e1}) and Lemma \ref{lm3}, $H(\cdot,v(\cdot),\nabla v(\cdot)) \in L^q(\Omega)$ for any $v \in \mathcal{A}$. Using Theorem \ref{th1}, there exist a unique function $w \in W^{2,q}(\Omega) $  and a constant $c_{\cc \label{t5c1}}>0$ which depends only on $n$, $p$, $\Omega$ and not on $v$ such that
	\begin{align}
	\begin{cases}
	L[\nabla v](w) =H(x,v,\nabla v) \quad
	\mathrm{in} \  \Omega, \\
	w-\phi \in W^{1,q}_0(\Omega), \\
	\|w\|_{W^{2,q}(\Omega)} \leq c_{\ref{t5c1}} 
	(\|G \|_{W^{1,p}(\Omega \times \R)}+ \| \phi\|_{W^{2,q}(\Omega)}). \label{t5e3}
	\end{cases}
	\end{align}
	By the Sobolev inequality and (\ref{t5e3}), we obtain
	\begin{align}
	\| w\|_{C^{1,\frac{1}{2}-\frac{n}{2q}}(\bar{\Omega})} &\leq c_{\ref{t5c2}} \| w\|_{C^{1,1-\frac{n}{q}}(\bar{\Omega})} \nonumber \\ &\leq c_{\ref{t5c3}} \| w\|_{W^{2,p}	(\Omega)}  \nonumber \\ 
	&\leq c_{\ref{t5c4}} (\|G \|_{W^{1,p}(\Omega \times \R)}+ \| \phi\|_{W^{2,q}(\Omega)}), \label{t5e4}
	\end{align}
	where $c_{\cc \label{t5c2}},c_{\cc \label{t5c3}},c_{\cc \label{t5c4}}>0$ depend only on $n$, $p$ and $\Omega$. 
	Suppose that
	\begin{align}
	\|G \|_{W^{1,p}(\Omega \times \R)}+ \| \phi\|_{W^{2,q}(\Omega)}  \leq c_{\ref{t5c4}}^{-1}=:\delta_2(n,p,\Omega). \label{t5e5}
	\end{align}
	Let a operator $T: \mathcal{A} \to \mathcal{A}$ be defined by $T(v)=w$ which satisfies (\ref{t5e3}).
	We show that $T$ is a compact and continuous mapping. For any sequence $\{v_m \}_{m \in \Na}$, we have
	$\sup_{m \in \Na} \| T(v_m) \|_{C^{1,1-\frac{n}{q}}(\bar{\Omega})} \leq c_{\ref{t5c2}}^{-1}  $ by (\ref{t5e4}, \ref{t5e5}).
	There exists a subsequence $\{T(v_k) \}_{k \in \Na} \subset    \{T(v_m) \}_{m \in \Na}$ which converges to
	a function $w_{\infty} \in C^1(\bar{\Omega})$ in the sense of $C^1(\bar{\Omega})$
	by the Ascoli-Arzel\`{a} theorem.
	We see that  $w_{\infty} \in C^{1,1-\frac{n}{q}}(\bar{\Omega})$ because
		\begin{align*}
		\frac{|\nabla {w_\infty}(x)-\nabla{w_\infty}(y)|}{|x-y|^{1-\frac{n}{q}}}
		= \lim_{k \to \infty} 
		\frac{
			|\nabla {T(v_k)}(x)-\nabla{T(v_k)}(y)|
			}{|x-y|^{1-\frac{n}{q}}} 
	\leq c_{\ref{t5c2}}^{-1}.
		\end{align*}
	 Let $\tilde{w}_k := T(v_k) -w_{\infty}$, and 
	$\tilde{w}_k$ converges to 0 in the sense of $C^1(\bar{\Omega})$.
	Then we have
	\begin{align}
	\frac{|\nabla \tilde{w}_k(x)-\nabla \tilde{w}_k(y)|}{|x-y|^{\frac{1}{2}-\frac{n}{2q}}}
	&\leq \left(\frac{|\nabla \tilde{w}_k(x)-\nabla \tilde{w}_k(y)|}{|x-y|^{1-\frac{n}{q}}} \right)^{\frac{1}{2}}
	|\nabla \tilde{w}_k(x)-\nabla \tilde{w}_k(y)|^{\frac{1}{2}} \nonumber \\
	&\leq 2c^{-\frac{1}{2}}_{\ref{t5c2}}  (2 \|\nabla \tilde{w}_k \|_{L^\infty(\Omega)})^{\frac{1}{2}} .
	\end{align}
	Hence, $\{T(v_k) \}_{k \in \Na}$ converges to a function $w_{\infty}$
	in the sense of $C^{1,\frac{1}{2}-\frac{n}{2q}}(\bar{\Omega})$, and the operator $T$ is a compact mapping.
	
		Suppose that $\{v_m \}_{m \in \Na}$ converges to $v$ in the sense of $C^{1,\frac{1}{2}-\frac{n}{2q}}(\bar{\Omega})$.\\
	 $ \sup_{m \in \Na}\|T(v_m)\|_{W^{2,q}(\Omega)}$ is bounded by (\ref{t5e4}, \ref{t5e5}). Hence, 
	there exists a subsequence $\{T(v_k) \}_{k \in \Na} \subset \{T(v_m) \}_{m \in \Na}$ which weakly converges to
	a function $w \in W^{2,q}(\Omega)$. We show $T(v)=w$, that is, 
	\[
	A_{ij}(\nabla v(x))w_{x_i x_j}(x) =H(x,v,\nabla v).
	\]
	For any $\phi \in C^\infty_0(\Omega)$, by the weak convergence and the H\"{o}lder  inequality,
	\begin{align}
	& \left|\int_{\Omega} \phi\{ A_{ij}(\nabla v)D_{ij}w -A_{ij}(\nabla v_k)D_{ij}(T(v_k)) \} \right| \nonumber \\
	\leq&\left| \int_{\Omega} \phi A_{ij}( \nabla v)(D_{ij}w - D_{ij}(T(v_k))) \right|  \nonumber \\
	& + \left| \int_{\Omega} \phi  D_{ij}(T(v_k)) (A_{ij}( \nabla v)-A_{ij}( \nabla v_k)) \right| \nonumber \\
	\leq&\left| \int_{\Omega} \phi A_{ij}( \nabla v)(D_{ij}w - D_{ij}(T(v_k))) \right|  \nonumber \\
	& +\|T(v_m)\|_{W^{2,q}(\Omega)} \|\phi 
	 (A_{ij}( \nabla v)-A_{ij}( \nabla v_k)) \|_{L^{\frac{q}{q-1}}(\Omega)} \nonumber \\
	 \to& 0 \quad (k \to \infty) . \label{t5e6}
	\end{align}
	By (\ref{t00e1}), we compute
	\begin{align}
	&|H(x,v_k(x),\nabla v_k(x))|  \nonumber \\
	 \leq& |G(x, v_k(x))-G(x,v(x))| +|G(x,v(x))| \nonumber \\
	 \leq& 	 \int_\R |G_t(x,t)| \ dt +|G(x,v(x))|.
	 \label{t5e7}
	\end{align}
	$\int_\R |G_t(\cdot,t)| \ dt +|G(\cdot,v(\cdot ))|$ is an integrable function
	by Lemma \ref{lm3} and Fubini's theorem. Since $H$ is a continuous function
	about $t$ and $z$, using the dominated convergence theorem,
	\begin{align}
	&\int_{\Omega} \phi\{H(x,v(x),\nabla v(x))-
	H(x,v_k(x),\nabla v_k(x)) \}\to 0 \quad (k \to \infty) . \label{t5e8}
	\end{align}

	By (\ref{t5e6}, \ref{t5e8}),
	\begin{align}
	&\int_{\Omega} \phi\{A_{ij}(\nabla v)D_{ij}w -H(x,v(x),\nabla v(x)) \} \nonumber \\
=&\lim_{k \to \infty} \int_{\Omega} \phi\{A_{ij}(\nabla v_k)D_{ij}(T(v_k)) -H(x,v_k(x),\nabla v_k(x)) \} \nonumber \\ =& 0.
	\end{align}
	Using the fundamental lemma of the calculus of variations,
	\begin{align*}
	&A_{ij}(x, \nabla v)D_{ij}w-H(x,v(x),\nabla v(x)) 
	=0 \quad a.e. \ x \in \Omega,
	\end{align*}
	and $T(v)=w$.
	Hence, $\{T(v_{m}) \}_{m \in \Na}$ weakly converges to $T(v)$ in $W^{2,q}(\Omega)$.
	By the compactness of $T$ and the uniqueness of limit, we can show $\{T(v_{m}) \}_{m \in \Na}$
	converges to $T(v)$ in $C^{1,\frac{1}{2}-\frac{n}{2q}}(\bar{\Omega})$, and $T$ is a continuous mapping.
	Using Theorem \ref{th4}, we obtain a function $u \in W^{2,q}(\Omega) $ satisfying 
	$u-\phi \in W^{1,q}_0(\Omega) $ and (\ref{t5e1}).
	\end{proof}

	\begin{proof}[Proof of Theorem \ref{mt2}]
		 We should show that there exists a function $\tilde{u} \in W^{2,q}(\Omega)$ 
		 such that
		\begin{align}
			A_{ij}(\nabla \tilde{u} + \nabla h) (\tilde{u}+h)_{x_i x_j} &=
			H(x,\tilde{u}+h,\nabla \tilde{u}+\nabla h) ,  \label{t6e1} \\
			\tilde{u}-\phi &\in W^{1,q}_0(\Omega) \label{t6e1b} \\
			\| \tilde{u} \|_{W^{2,q}(\Omega)} &<\e. \label{t6e1c}
		\end{align}
		Using the minimal surface equation (\ref{t6e1a}) for $h$, we convert (\ref{t6e1}) as
		\begin{align}
		&A_{ij}(\nabla \tilde{u} + \nabla h)\tilde{u}_{x_i x_j}  +
		\frac{h_{x_i x_j}}{(1+|\nabla \tilde{u} + \nabla h|^2)^{\frac{3}{2}}}
		((|\nabla \tilde{u}|^2 +\nabla \tilde{u} \cdot \nabla h)\delta_{ij} \nonumber \\&
		- \tilde{u}_{x_i}\tilde{u}_{x_j}-\tilde{u}_{x_i}h_{x_j}
		-\tilde{u}_{x_j}h_{x_i}) \nonumber \\
		=&H(x,\tilde{u}+h,\nabla \tilde{u}+\nabla h) .
		\end{align}
		Define
		\begin{align}
		\mathcal{A}:=\{v \in C^{1,\frac{1}{2}-\frac{n}{2q}}(\bar{\Omega}) ; \|v\|_{C^{1,\frac{1}{2}-\frac{n}{2q}(\bar{\Omega})}} \leq \e  \}.
		\end{align}
Using \cite[Theorem 9.15]{G83}, for any $v \in \mathcal{A}$, there exists a unique function $w \in W^{2,q}(\Omega) $ such that $w-\phi \in W^{1,q}_0(\Omega)$
 and
\begin{align}
&A_{ij}(\nabla v + \nabla h)w_{x_i x_j}  +
\frac{h_{x_i x_j}}{(1+|\nabla v + \nabla h|^2)^{\frac{3}{2}}}
((\nabla v\cdot \nabla w +\nabla w \cdot \nabla h)\delta_{ij} \nonumber \\&
- v_{x_i}w_{x_j}-w_{x_i}h_{x_j}
-w_{x_j}h_{x_i}) \nonumber \\
=&H(x,v+h,\nabla v+\nabla h).\label{t6e3}
\end{align}
Define
\begin{align*}
B(\nabla v) \cdot \nabla w:=&\frac{h_{x_i x_j}}{(1+|\nabla v + \nabla h|^2)^{\frac{3}{2}}}
((\nabla v\cdot \nabla w +\nabla w \cdot \nabla h)\delta_{ij} \\
&- v_{x_i}w_{x_j}-w_{x_i}h_{x_j}
-w_{x_j}h_{x_i}).
\end{align*}
Here, $B:\R^n \to \R^n$ is a continuous function.
By Lemma \ref{lm3}, a similar argument of Theorem \ref{th1} and the Sobolev inequality,
there exists a constant $c_{\cc \label{t6c2}}>0$ which depends only on $n$, $p$, $\Omega$, $\e$ and $\| h\|_{W^{2, \infty}	(\Omega)} $ such that
\begin{align}
\| w\|_{C^{1,\frac{1}{2}-\frac{n}{2q}(\bar{\Omega})}} \leq c_{\ref{t6c2}}(\|G \|_{W^{1,p}(\Omega \times \R)}+ \| \phi\|_{W^{2,q}(\Omega)}) . \label{t6e4}
\end{align}
Suppose that we have
\begin{align}
	\|G \|_{W^{1,p}(\Omega \times \R)}+ \| \phi\|_{W^{2,q}(\Omega)}  \leq c_{\ref{t6c2}}^{-1} \e:= \delta_1 . \label{t6e5}
\end{align}
Let a operator $T: \mathcal{A} \to \mathcal{A}$ be defined by $T(v)=w$ which satisfies
$w-\phi \in W^{1,q}_0(\Omega)$ and (\ref{t6e3}).
The compactness of $T$ can be proved by the argument of Theorem \ref{mt1}.

Suppose that $\{v_m \}_{m \in \Na} \subset \mathcal{A} $ converges to $v$ in the sense of $C^{1,\frac{1}{2}-\frac{n}{2q}}(\bar{\Omega})$.
Then there exists a subsequence $\{T(v_k) \}_{k \in \Na} \subset \{T(v_m) \}_{m \in \Na}$ which weakly converges to
a function $w \in W^{2,q}(\Omega)$. For any $\phi \in C^\infty_0(\Omega)$, 
\begin{align}
&\int_{\Omega} \phi\{ B(\nabla v) \cdot \nabla w -B(\nabla v_k) \cdot \nabla T(v_k) \} \nonumber \\
=& \int_{\Omega} \phi B(\nabla v) \cdot(\nabla w - \nabla(T(v_k)))  \nonumber \\
&+ \int_{\Omega} \phi  \nabla(T(v_k)) \cdot (B(\nabla v)-B(\nabla v_k)) \nonumber \\
 \to& 0 \quad (k \to \infty) , \label{t6e6}
\end{align}
since $B$ is a continuous function and $T(v_k) $ converges weakly to $w$.
By (\ref{t6e6}) and the argument of Theorem \ref{mt1}, we can show that $T$ is 
a continuous mapping. 
Using Theorem \ref{th4}, we obtain a function $\tilde{u} \in W^{2,q}(\Omega) $ satisfying (\ref{t6e1}, \ref{t6e1b}). Moreover, $\tilde{u}$ satisfies (\ref{t6e1c}) by (\ref{t6e4}, 
\ref{t6e5}). Define $u:=\tilde{u}+h$. Then $u$ satisfies $u-h-\phi \in
W^{1,q}_0(\Omega)$ and (\ref{t6e1d}, \ref{t6e1e}), and the proof is complete.
\end{proof}
	
	\begin{cor} \label{maincor}
		Suppose that we are given
		$f =(f_1, \cdots,f_{n+1})\in W^{1,p}(\Omega \times \R ;\R^{n+1} )$ where $\frac{n+1}{2}<p< n+1$ and
		$\phi \in W^{2,q}(\Omega)$ where $q=\frac{np}{n+1-p}$.
			Let $\e>0$ be arbitrary. Suppose $h \in W^{2,\infty}(\Omega)$ satisfies
			the minimal surface equation, that is,
			\begin{align}
				\di\left( \frac{\nabla h}{\sqrt{1+|\nabla h|^2}} \right) =0.
			\end{align}
       Let $\delta_1>0$ be the constant as in Theorem \ref{mt2}.
		If
		\begin{align}
			\sum_{i=1}^{n+1} \|f_i\|_{W^{1,p}(\Omega \times \R )} +\|\phi \|_{W^{2,q}(\Omega)} \leq \delta_1,
		\end{align}
		then there exists a function $u \in W^{2,q}(\Omega)$ 
		such that $ u-h-\phi \in W^{1,q}_0(\Omega)$ and
		\begin{align}
			\di\left( \frac{\nabla u}{\sqrt{1+|\nabla u|^2}} \right) &=
			\nu(\nabla u(x))\cdot f(x,u(x))  \quad
			\mathrm{in} \  \Omega, \\
\|u-h\|_{W^{2,q}(\Omega)}&<\e.
		\end{align}
	\end{cor}
	\begin{proof}
		Define
		\[
		H(x,t,z):= \nu(z)\cdot f(x,t).
		\]
		By $f \in W^{1,p}(\Omega \times \R ;\R^{n+1} )$
		, for a.e. $x \in \Omega$, $f(x,\cdot)$ is an
		absolutely continuous function. Hence $H(x,\cdot,\cdot)$ is a continuous function for almost every $x \in \Omega$. We have
		\[
		|H(x,t,z)| \leq \sum_{i=1}^{n+1} |f_i(x,t)| \quad \mathrm{for} \   a.e. \ x \in \Omega, \ 
		\mathrm{any} \  (t,z) \in \R \times \R^{n},
		\]
		and $\sum_{i=1}^{n+1} |f_i(x,t)| \in W^{1,p}(\Omega \times \R)$.
		By the Minkowski inequality,
		\[
		\|\sum_{i=1}^{n+1} |f_i(x,t)| \|_{W^{1,p}(\Omega \times \R)}
		\leq \sum_{i=1}^{n+1} \|f_i\|_{W^{1,p}(\Omega \times \R )}.
		\]
		Define
		\[
		G(x,t):=\sum_{i=1}^{n+1} |f_i(x,t)|.
		\]
		Then $H$ and $G$ satisfy the assumption of Theorem \ref{it2},
		 and this corollary follows.
	\end{proof}

\begin{rem}
	The uniqueness of solutions follows immediately using \cite[Theorem 10.2]{G83}.
	Under the assumptions of Theorem \ref{mt2},
	if we additionally assume that
	$H$ is non-increasing 
	in $t$ for each $(x,z) \in \Omega \times \R^n$ and
	continuously differentiable with respect to the $z$ variables 
	in $\Omega \times \R \times \R^n$, then the solution is unique in $W^{2,q}(\Omega)$.
\end{rem}

\end{document}